\documentclass{amsart}
\usepackage[utf8]{inputenc}
\usepackage{amsmath,amssymb,amsthm,amsfonts,mathrsfs,enumitem,array,comment}
\usepackage{color}
\usepackage{dsfont}

\DeclareMathOperator\Aut{Aut}
\DeclareMathOperator\ch{char}
\DeclareMathOperator\End{End}
\DeclareMathOperator\Ext{Ext}
\DeclareMathOperator\GKdim{GKdim}

\DeclareMathOperator\gr{gr}

\DeclareMathOperator\id{id}
\DeclareMathOperator\im{Im}
\DeclareMathOperator\Ker{Ker}
\DeclareMathOperator\Kdim{Kdim}

\DeclareMathOperator\SL{SL}
\DeclareMathOperator\Sp{Sp}

\newcommand\fg{\mathfrak g}
\newcommand\fh{\mathfrak h}

\newcommand\pl{\mathfrak{pl}}

\newcommand\inv{^{-1}}
\newcommand\iso{\cong}
\newcommand{\x}{\text}
\newcommand\kk{\mathds{k}}
\newcommand\p{\mathsf{p}}
\newcommand\tensor{\otimes}

\newcommand\bc{\mathbf c}

\newcommand\bq{\mathbf q}
\newcommand\bH{\mathbf H}

\newcommand\cF{\mathcal F}
\newcommand\cS{\mathcal S}

\newcommand\NN{\mathbb N}
\newcommand\ZZ{\mathbb Z}

\newcommand\jp{\mathcal{J}}
\newcommand\qp{\mathcal{O}_q(\kk^2)}
\newcommand\fwa{A_1(\kk)}
\newcommand\wa{A_1^q(\kk)}
\newcommand\qwa{\mathcal{A}_n^{\bq,\Gamma}}

\newtheorem{lemma}[equation]{Lemma}

\newtheorem{theorem}[equation]{Theorem}
\newtheorem{corollary}[equation]{Corollary}
\newtheorem{definition}[equation]{Definition}
\newtheorem{remark}[equation]{Remark}
\newtheorem{example}[equation]{Example}
\newtheorem*{main}{Main Theorem}

\newcommand{\grp}[1]{\langle{#1}\rangle}

\newcommand\jason{\textcolor{blue}}

\setlist[enumerate,1]{label=(\arabic*), ref=\arabic*}

\title{Congenial Algebras: Extensions and Examples}

\author[Gaddis]{Jason Gaddis}
\address{Miami University, Department of Mathematics, 301 S. Patterson Ave., Oxford, Ohio 45056} 
\email{gaddisj@miamioh.edu}

\author[Yee]{Daniel Yee}
\address{Bradley University, Department of Mathematics, 1501 W. Bradley Ave., Peoria, IL 61625}
\email{dyee@bradley.edu}


\begin{document}

\begin{abstract}
We study the congeniality property of algebras, as defined by Bao, He, and Zhang,
in order to establish a version of Auslander's theorem for various families of filtered algebras.
It is shown that the property is preserved under homomorphic images and tensor products
under some mild conditions.
Examples of congenial algebras in this paper include 
enveloping algebras of Lie superalgebras, iterated differential operator rings, 
quantized Weyl algebras, down-up algebras, and symplectic reflection algebras.
\end{abstract}

\maketitle


An important result of Auslander \cite{A} shows that 
if $V$ is a finite-dimensional vector space over an algebraically closed field $\kk$ of characteristic zero, and $G$ is a finite group of
automorphisms acting linearly on $\kk[V]$ with no nontrivial reflections (i.e., $G$ is a {\sf small} group), then there is an isomorphism of graded algebras $\kk[V]\# G \rightarrow \End_{\kk[V]^G} \kk[V]$.
There has been much work done in extending this result to
the noncommutative setting, either by replacing $\kk[V]$ by
a suitable noncommutative algebra, replacing $G$ with a Hopf algebra $H$, or both.

Recent work of Bao, He, and Zhang introduces the pertinency invariant
as a way to test whether an algebra $A$ and a Hopf algebra $H$ acting on $A$
satisfy the conclusion of Auslander's Theorem \cite{BHZ2,BHZ1}.
The general theme of their results is that, for a suitable pair,
this holds if and only if the pertinency is at least two.
In \cite{BHZ2} it is shown that a class of filtered algebras known as {\it congenial algebras}, 
along with certain groups of filtered automorphisms, are sufficiently suitable.
Furthermore, they prove that the enveloping algebra of a finite-dimensional Lie algebra is congenial. 
The authors state that there are `ample examples of congenial algebras' and
part of our goal is to better understand what algebras satisfy this condition.
We prove the following theorem via various results in this paper.

\begin{main}
The following algebras are congenial with respect to some filtration.
\begin{enumerate}[label=\arabic*.,leftmargin=*]
\item Certain images of congenial algebras under filtration-preserving homomorphisms (Theorem \ref{thm.map}).
\item Certain tensor products of congenial algebras (Theorem \ref{thm.tensor}).
\item The enveloping algebra of a finite-dimensional Lie superalgebra (Theorem \ref{thm.super}).
\item An iterated Ore extension of derivation type over $\kk$
(Theorem \ref{thm.iterated}).
\item A quantized Weyl algebra (Theorem \ref{thm.qweyl}).
\item A noetherian down-up algebra $A(\alpha,\beta,\gamma)$
assuming the roots of $x^2-\alpha x -\beta=0$ are roots of unity
of order at least two
(Theorem \ref{thm.du}).
\item A symplectic reflection algebra (Theorem \ref{thm.symplectic}).
\end{enumerate}
\end{main}

A common theme to our examples is that, while these algebras are most commonly
defined over a field, the definition works over a suitable commutative ring.
Hence, we are able in these cases to define an {\it order} of the algebra \cite{BHZ2},
and use reduction mod $p$ techniques.

In addition, we establish when possible families of groups acting on
classes of algebras above for which the conclusion of Auslander's Theorem holds.
Notably, we prove that $A\#G \iso \End_{A^G} A$ for any filtered 
Artin-Schelter regular algebra of global dimension 2 and $G$ a finite group of filtered automorphisms acting with trivial homological determinant (Theorem \ref{thm.filtAS}).

\subsection*{Background}
Given an algebra $A$ and a group $G$ acting as automorphisms on $A$,
the {\sf skew group algebra} $A\# G$ is defined to be the $\kk$-vector space $A \tensor \kk G$ with multiplication,
\[ (a \# g)(b \# h) = a g(b) \# gh \quad\text{for all $a,b \in A$, $g,h \in G$.}\]
For a filtered algebra $A$ and a finite group $G$ acting as filtered automorphisms on $A$, 
the {\sf Auslander map} is given by
\begin{align*}
\gamma_{A,G} : A \# G &\to \End_{A^G}(A) \\
       a \# g &\mapsto
       \left(\begin{matrix}A &\to & A \\ b & \mapsto & ag(b)\end{matrix}\right).
\end{align*}
Suppose $A=\kk[x_1,\hdots,x_p]$ and $G$ is a finite group that acts linearly on $A$ without reflections,
then a theorem of Auslander asserts that $\gamma_{A,G}$ is an isomorphism \cite{A}.

Let $A$ be an affine algebra and $G$ a finite group acting on $A$.
The {\sf pertinency} of the $G$-action on $A$ is defined to be
\[ \p(A,G) = \GKdim A - \GKdim (A\# G)/(f_G)\]
where $(f_G)$ is the two sided ideal of $A\#G$ generated by
$f_G = \sum_{g \in G} 1\# g$ and $\GKdim$ is the Gelfand-Kirillov (GK) dimension.
It is possible to define pertinency in terms of any dimension function on right $A$-modules, such as Krull dimension ($\Kdim$), but GK dimension is sufficient for our purpose.
It is also possible to define all of the above in terms of actions by semisimple Hopf algebras, however our focus will be on group actions.

An algebra $A$ is said to be {\sf Cohen-Macaulay} (CM) if $\GKdim(A)=d \in \NN$ and
$j(M)+\GKdim(M)=\GKdim(A)$ for every finitely generated right $A$-module $M\neq 0$.
Here $j(M)=\min\{i : \Ext_A^i(M,A)\neq 0\}$, or $\infty$ if no such $i$ exists, denotes the {\sf grade} of the module $M$.
If one replaces $\GKdim$ with $\Kdim$, then we say $A$ is {\sf Kdim-CM}.

Under suitable conditions, the Auslander map is an isomorphism (for a particular
pair $(A,G)$) if and only if $\p(A,G) \geq 2$.
In particular, by various results in \cite{BHZ2,BHZ1}, this applies when
\begin{enumerate}
\item $A$ is noetherian, connected graded AS regular, and CM of GK dimension at least two, and $G$ acts linearly;
\item $A$ is a noetherian PI Kdim-CM algebra with $\Kdim(A) \geq 2$;
\item $A$ is {\it congenial} and $G$ preserves the filtration on $A$.
\end{enumerate}

Recall that $\kk$ is an algebraically closed field of characteristic zero.
Let $\kk_0=\ZZ$ or $F_p := \ZZ/(p)$, $p \in \kk$ prime.
Also recall that an $R$-algebra $A$ is {\sf strongly noetherian} if $A \tensor_R C$ is noetherian 
for any commutative noetherian $R$-algebra $C$ and {\sf universally noetherian}
if we can drop the commutative hypothesis on $C$.


\begin{definition}
Let A be an algebra over $\kk$ and let D be a $\kk_0$-affine subalgebra of $\kk$.
A $D$-subalgebra $A_D$ of $A$ is called an {\sf order} of $A$ if $A_D$ is free over $D$ and $A_D \tensor_D \kk = A$.
\end{definition}

\begin{definition}
\label{defn.congenial}
A $\kk$-algebra $A$ is called {\sf congenial} if the following hold.
\begin{enumerate}
\item $A$ is a locally finite filtered algebra with an $\NN$-filtration $\cF$.
\item $A$ has an order $A_D$ where $D$ is a $\kk_0$-affine subalgebra of $\kk$ such that $A_D$ is a locally finite filtered algebra over $D$ with the induced filtration, still denoted by $\cF$.
\item The associated graded ring $\gr_\cF A_D$ is an order of $\gr_\cF A$.
\item $gr_\cF A_D$ is a strongly noetherian locally finite graded algebra over $D$. 
\item If $F$ is a factor ring of $D$ and is a finite field, then $A_D \tensor_D F$ 
is an affine noetherian PI algebra over $F$.
\end{enumerate}
\end{definition}

The original definition of congeniality requires $A_D$ to be noetherian but this is implied by (4).
As the referee observed, the original noetherian assumption in (1) is also unnecessary by (3) and (4).
As stated in \cite{BHZ2}, one can always replace $D$ in Definition \ref{defn.congenial}
with a $D$-affine subalgebra $D' \subset \kk$.
We use this fact freely throughout without comment.

It is important to note that the definition of congenial is relative to some filtration
on the algebra $A$ and the results of \cite{BHZ2} only apply to groups acting on $A$
that preserve the filtration.
Hence, when we state that an algebra is congenial, we either define explicitly
a filtration or assume that it is clear from the context.
For example, the filtration that we use on Lie superalgebras is the standard PBW filtration.

\subsection*{Congenial extensions}

We consider how the congeniality property behaves under both homomorphisms and tensor products.

\begin{theorem}\label{thm.map}
Let $\psi:A\rightarrow B$ be a surjective $\kk$-algebra homomorphism of filtered algebras that preserves the filtration on the congenial algebra $A$ with order $A_D$.
If $B_D=\psi(A_D)$ and $\gr(B)_D$ are free as $D$-modules,
then $B$ is congenial
\end{theorem}
\begin{proof} 

(1) It is clear that $B$ is a locally finite filtered algebra.

(2) Since $B \iso A/\ker\psi$, then we may regard $B$ (and $B_D$) as a $D$-module. 
Note $B_D$ is free by hypothesis.
Naturally there is a surjective $D$-algebra homomorphism $\psi_D:A_D\rightarrow B_D$ induced by $\psi$. Thus $B_D\otimes_D\kk=\psi_D(A_D)\otimes_D\kk=\psi(A)=B$, hence $B_D$ is an order of $B$.

(3) The map $\psi$ naturally induces a surjective $\kk$-algebra homomorphism $\gr\psi:\gr A\rightarrow \gr B$. A similar argument to (2) shows that $\gr B_D$ is an order of $\gr B$.

(4) For any commutative noetherian $D$-algebra $C$, $\gr A_D\otimes_DC$ is noetherian. Applying the map $\gr \psi\otimes\id_C$ implies that $\gr B_D\otimes_DC$ is also noetherian.

(5) Let $F$ be a factor ring of $D$ and finite field. Applying the $F$-algebra homomorphism $\psi_D\otimes\id_F:A_D\otimes_DF\rightarrow B_D\otimes_DF$ shows that $B_D\otimes_DF$ affine noetherian PI over $F$, since $A_D\otimes_DF$ is affine noetherian PI.
\end{proof}

The hypothesis that $B_D$ and $\gr(B)_D$ be free is satisfied automatically in certain cases, such as when $D$ is a PID.
For example, the first Weyl algebra $\fwa$ is the homomorphic image of the enveloping algebra $U(\fh)$ where $\fh$ is the Heisenberg Lie algebra.
In this case, we may take $D=\ZZ$. 
Both $(\fwa)_\ZZ=A_1(\ZZ)$ and $\gr(\wa)_\ZZ=\gr(A_1(\ZZ))=\ZZ[x,y]$ (under the standard filtration) are free $\ZZ$-modules. Hence, $\fwa$ is congenial.
We generalize this in Corollary \ref{cor.pbw} but first we consider a short result on tensor products.

\begin{corollary}
\label{cor.tensor}
Suppose $A_1$ and $A_2$ are $\kk$-algebras such that $A_1\otimes_\kk A_2$ is congenial with corresponding $\kk_0$-affine subalgebra $D$
If $(A_1)_D$ (resp. $(A_2)_D$) and $\gr(A_1)_D$ (resp. $\gr(A_2)_D$) are free as a $D$-module, then $A_1$ (resp. $A_2$) is congenial.
\end{corollary}
\begin{proof}
This follows from Theorem \ref{thm.map} with the maps $a_1\otimes a_2\mapsto a_1\otimes1$ and $a_1\otimes a_2\mapsto1\otimes a_2$.
\end{proof}

\begin{corollary}
\label{cor.pbw}
If $A$ be an affine, connected filtered algebra generated by $A_1$
such that $\gr A$ is a polynomial algebra, then $A$ is congenial.
Moreover, if $G$ is a finite small subgroup of linear automorphisms on $A$, then $\p(A,G) \geq 2$ and so $A\# G \iso \End_{A^G} A$.
\end{corollary}
\begin{proof}
Let $A_1=\{a_1,\hdots,a_n\}$.
Then for each $i \neq j$, $[a_i,a_j] = c_{ij0}+\sum_{k=1}^n c_{ijk} a_k$.
Define $D$ as the algebra over $\ZZ$ generated by the nonzero $c_{ijk}$
and let $A_D$ be $A$ but with coefficients restricted to $D$.
It is clear that $A_D$ is free over $D$ with basis the set of standard monomials.

Under the given hypotheses, $A$ is almost commutative \cite{MR},
and hence it is the homomorphic image of an enveloping algebra $U(\fg)$ for some finite-dimensional Lie algebra $\fg$ \cite[Proposition 4.3]{MR}.
The enveloping algebra $U(\fg)$ is congenial by \cite[Lemma 4.11]{BHZ2} and we can choose $D$ as the corresponding $\kk_0$-affine subalgebra of $\kk$.
Thus, $A$ is congenial by Theorem \ref{thm.map}.

As $\gr A$ is a polynomial algebra, $A$ is CM and so we may apply \cite[Theorem 4.10]{BHZ2}.
\end{proof}

Next we prove a partial converse to Corollary \ref{cor.tensor}.

\begin{theorem}
\label{thm.tensor}
If $A_1,A_2$ are congenial $\kk$-algebras and if
$\gr(A_1)_{D_1}$ is universally noetherian, 
then $A_1 \tensor_{\kk} A_2$ is congenial.
\end{theorem}

\begin{proof}
(1) 
Let $\cF_1$ and $\cF_2$ be the respective filtrations of $A_1$ and $A_2$.
Define a filtration $\cF$ on $A=A_1 \tensor_{\kk} A_2$ by declaring that
$\cF_n(A)$ is generated (as a vector space) by those tensors $a_1 \tensor a_2$
such that $a_1 \in \cF_r$, $a_2 \in \cF_s$, and $r+s \leq n$.
It is clear that $\cF$ is locally finite since $\cF_1$ and $\cF_2$ are.

(2) Let $D=\grp{D_1 \cup D_2}$. As $D_1$ and $D_2$ are both $\kk_0$-affine
then so is $D$.
Since $D$ is finitely generated over both $D_1$ and $D_2$,
then the algebras $A_1$ and $A_2$ over $D$, denoted $(A_1)_D$ and $(A_2)_D$, are orders of $A_1$ and $A_2$, respectively, and both $A_1$ and $A_2$ are congenial with respect to $D$ and the induced filtration.
Moreover, 
\[(A_1 \tensor_\kk A_2)_D = (A_1 \tensor_D A_2)_D = (A_1)_D \tensor_D (A_2)_D.\]
It follows that $(A_1 \tensor_\kk A_2)_D \tensor_D \kk = A_1 \tensor_\kk A_2$ and so 
$(A_1 \tensor_\kk A_2)_D$ is an order for $A_1 \tensor_\kk A_2$.

(3) It is clear from the filtration that
$\gr_\cF (A_1 \tensor_\kk A_2) = \gr_{\cF_1} A_1 \tensor_\kk \gr_{\cF_2} A_2$ and 
\begin{align}
\label{ex.gr}
\gr_\cF (A_1 \tensor_\kk A_2)_D = \gr_{\cF_1} (A_1)_D \tensor_D \gr_{\cF_2} (A_2)_D.\tag{$\star$}
\end{align}
The argument in (2) shows that $\gr_\cF (A_1 \tensor A_2)_D$ is an
order of $\gr_\cF (A_1 \tensor A_2)$.

(4) Let $C$ be a commutative noetherian $D$-algebra.
Then $\gr_{\cF_2} (A_2)_D \tensor_D C$ is noetherian because $\gr_{\cF_2} (A_2)_D$
is assumed to be strongly noetherian. 
Hence, because $\gr_{\cF_1} (A_1)_{D_1}$ is universally noetherian and so is $\gr_{\cF_1}(A_1)_D$, then by \eqref{ex.gr}, 
\[
\gr_\cF (A_1 \tensor_\kk A_2)_D \tensor_D C
	= \gr_{\cF_1} (A_1)_D \tensor_D (\gr_{\cF_2} (A_2)_D \tensor_D C)
\]
is noetherian. Thus, $\gr_\cF (A_1 \tensor A_2)_D$ is strongly noetherian.

(5) Let $F$ be a finite field and a factor ring of $D$. It is obvious that the tensor product of two affine algebras is affine. Consider the map 
\begin{align*}
(A_1\otimes_\kk A_2)_D\otimes_DF &\xrightarrow{\sim} [(A_1)_D\otimes_DF]\otimes_D[(A_2)_D\otimes_DF] \\
a_1\otimes a_2\otimes t &\mapsto (a_1\otimes 1)\otimes(a_2\otimes t).
\end{align*}
It is clear that this map is an injective $D$-algebra homomorphism. 
Since both $(A_1)_D\otimes_DF$ and $(A_2)_D\otimes_DF$ are PI over $F$, 
then $[(A_1)_D\otimes_DF]\otimes_D[(A_2)_D\otimes_DF]$ is PI, 
and thus $(A_1\otimes A_2)\otimes_DF$ is PI over $F$. 
As $A_1$ is universally noetherian and $(A_2)_D\otimes_DF$ is noetherian by hypothesis,
then $(A_1\otimes_\kk A_2)\otimes_DF=A_1\otimes_{\kk}(A_2\otimes_DF)$ is noetherian.
\end{proof}

\begin{corollary}
Let $A_1,A_2$ be congenial algebras satisfying the conditions of Theorem \ref{thm.tensor}.
Suppose $G$ is a finite subgroup of automorphisms respecting the filtration on $A_1$
such that $\p(A_1,G) \geq 2$.
Then we may identify $G$ within $\Aut(A_1 \tensor A_2)$ and $\p(A_1 \tensor A_2,G) \geq 2$.
Thus, the Auslander map is an isomorphism for $(A_1 \tensor A_2,G)$.
\end{corollary}

\begin{proof}
This follows from Theorem \ref{thm.tensor}, \cite[Theorem 3.12]{GKMW},
and \cite[Theorem 4.10]{BHZ2}.
\end{proof}

We now consider various families of algebras with (nice) associated graded rings
for which we can prove the congeniality property and a version of Auslander's theorem.


\subsection*{Enveloping algebras of Lie superalgebras}
A {\sf Lie superalgebra} is a $\ZZ_2$-graded $\kk$-vector space $L=L_0 \oplus L_1$ with
a bilinear (super)bracket $[,]$ satisfying
\begin{enumerate}
\item $[L_\alpha,L_\beta] \subset L_{\alpha+\beta}$ (index taken modulo 2),
\item $[x,y] = (-1)^{|x||y|} [y,x]$ (super skew-symmetry), and
\item $(-1)^{|x||z|}[x,[y,z]] + (-1)^{|y||x|}[y,[z,x]] + (-1)^{|z||y|}[z,[x,y]]=0$ (super Jacobi identity),
\end{enumerate}
where $x,y,z$ are homogeneous in $\ZZ_2$-degree and $|x|$ represents the degree of $x$.
The {\sf universal enveloping algebra of L} is the quotient of the tensor algebra $T(L)$
by the ideal of relations generated by
$x \tensor y - (-1)^{|x||y|} y \tensor x - [x,y]$ for all $x,y \in L$.

\begin{lemma}
\label{lem.super1}
If $U(L)$ is the enveloping algebra of a finite-dimensional Lie superalgebra $L$ 
over a field $F$ of characteristic $p>0$, then $U(L)$ is PI.
\end{lemma}

\begin{proof}
Since $L_0$ is an (ordinary) Lie algebra, then $U(L_0)$ is module-finite over its center \cite[Lemma VI.5]{Ja1}.
By \cite[Proposition 2]{behr}, $U(L)$ is module-finite over $U(L_0)$ and so the result follows.
\end{proof}

\begin{theorem}
\label{thm.super}
The enveloping algebra of a finite-dimensional Lie superalgebra $L$ is congenial.
\end{theorem}


\begin{proof}
(1) Set $A=U(L)$.
It follows from the PBW filtration that $A$ is a locally finite filtered algebra.

(2) Let $\{x_1,...,x_n\}$ be a basis for $L$. Write $[x_i,x_j]=\sum_{p=1}^n\alpha_p(i,j)x_p$ where each $\alpha_p(i,j)\in \kk$. Let $D$ be the subring of $\kk$ generated by the set $\{\alpha_p(i,j):i,j,p\leq n\}$.

Define $L_D$ to be the Lie superalgebra over $D$ with $L_D\otimes_D \kk=L$ and
\[ A_D:=U(L_D)=D\langle x_1,...,x_n\rangle/(x_ix_j-(-1)^{|x_i||x_j|}x_jx_i-\sum_{p=1}^n\alpha_p(i,j)x_p).\]
It follows that $A_D\otimes_D\kk=A$.
Let $\cF$ be the PBW filtration on $A$.
Then $\cF$ descends to a filtration on $A_D$, whence
$A_D$ is a locally finite filtered algebra over $D$.

(3) Furthermore, we see that $\gr_{\cF}A_D\otimes_D \kk=\gr A$, since
\[ \gr_{\cF}A_D=D[x_1,...,x_t]\otimes_D\bigwedge_D(x_{t+1},...,x_p),\]
where $\bigwedge_D(x_{t+1},...,x_p)$ is the exterior algebra over $D$, and $x_1,..,x_t$ is a basis for the (ordinary) Lie subalgebra $L_0$.

(4) Since $\gr_{\cF}A_D$ is PI noetherian, then $\gr_{\cF}A_D$ is strongly noetherian \cite[Proposition 4.4]{ASZ}.

(5) Let $F$ be a factor ring of $D$ that is a finite field.
Then $A_D \tensor F = U(L_D \tensor_D F)$ and so $A_D \tensor_D F$ is affine noetherian. 
The PI condition now follows from Lemma \ref{lem.super1}.
\end{proof}

\begin{corollary}
\label{cor.super}
Let $L$ be a finite-dimensional Lie superalgebra, $R=U(L)$, and let
$G$ be a finite group of automorphisms of $R$ respecting the filtration of $R$.
If $G$ restricts to a small subgroup of $\Aut_{\mathrm{Lie}}(L_0)$,
then there is a natural isomorphism of algebras $R\#G \iso \End_{R^G}(R)$.
\end{corollary}

\begin{proof}
By Theorem \ref{thm.super}, $R$ is congenial. 
It is well-known that $R$ is CM, and hence by \cite[Theorem 4.10]{BHZ2},
it is sufficient to show that $\p(R,G) \geq 2$.
The $G$-action on $R$ induces a $G$-action on 
$\gr(R) \iso \kk[x_1,...,x_t]\tensor_{\kk} \bigwedge_{\kk}(x_{t+1},...,x_p)$.
In particular, because $G \subset \Aut_{\mathrm{Lie}}(L_0)$, then $G$
restricts to an action on $\kk[x_1,...,x_t]$ such that
$\p(\kk[x_1,...,x_t],G) \geq 2$ by \cite[Theorem 4.12]{BHZ2}.
It follows from the proof of \cite[Theorem 3.13]{GKMW} that
$\p(\gr(R),G) = \p(\kk[x_1,...,x_t],G)$.
Now by \cite[Proposition 3.3]{BHZ2}, $\p(R,G) \geq \p(\gr(R),G) \geq 2$.
\end{proof}

Here is one example to which Theorem \ref{thm.super} applies.

\begin{example}
\label{ex.super}
Consider the Lie superalgebra $L=\pl(1|1)$ generated by
\[
x_1 = \begin{pmatrix}1 & 0 \\ 0 & 1\end{pmatrix}, \quad
x_2 = \begin{pmatrix}1 & 0 \\ 0 & 0\end{pmatrix}, \quad
y_1 = \begin{pmatrix}0& 1 \\ 0 & 0\end{pmatrix}, \quad
y_2 = \begin{pmatrix}0 & 0 \\ 1 & 0\end{pmatrix},
\]
where $\{x_1,x_2\}$ generate $L_0$ and $\{y_1,y_2\}$ generate $L_1$.
The element $x_1$ is central in $U(L)$ while $y_1,y_2$ are nilpotent with $y_1^2=y_2^2=0$.
The remaining relations in $U(L)$ are
$y_1y_2+y_2y_1=x_1$, $x_2y_1-y_1x_2=y_1$, and $x_2y_2-y_2x_2=-y_2$.
Define a family of automorphisms parameterized by $\lambda\in \kk^\times$ by
\[
\phi(x_1)=\lambda x_1, \quad \phi(x_2)=-x_2, \quad \phi(y_1)=y_2, \quad \phi(y_2)=\lambda y_1.
\]
If $|\lambda|<\infty$, then $\phi$ restricts to a finite subgroup of $U(L_0)$
and if in addition $\lambda\neq 1$, then the group $G=\langle\phi\rangle$ is small
when restricted to $\gr U(L_0)=\kk[x_1,x_2]$.
It follows that the Auslander map is an isomorphism for the pair $(U(L),G)$.
\end{example}

\subsection*{Iterated differential operator rings}
Let $R$ be an algebra and $\delta$ a derivation on $R$.
That is, $\delta$ obeys the Leibniz rule, $\delta(rs)=r\delta(s)+\delta(r)s$
for all $r,s \in R$.
The (over)ring $R[x;\delta]$ generated by $R$ and $x$ with the rule
$xr=rx+\delta(r)$ is called a {\sf differential operator ring over $R$}.

Consider an iterated differential operator ring $A=\kk[x_1][x_2;\delta_2]\cdots[x_n;\delta_n]$.
We define a filtration $\cF$ on $A$ by setting
$\deg(x_1)=1$ and $\deg(x_k)=\max\{\deg(\delta(x_j)) : j < k\}$
for $k>1$.

\begin{theorem}
\label{thm.iterated}
If $A$ is an iterated Ore extension of derivation type over $\kk$, then $A$ is congenial.
\end{theorem}

\begin{proof}
(1) Using the filtration $\cF$ defined above, we see that $A$ is a locally finite filtered algebra. 

(2) Set $D$ to be the $\kk_0$-affine subalgebra of $\kk$ defined by
\[ D = \kk_0 \cup \{\text{coefficients appearing in nonzero summands of $\delta_i(x_j)$ for $i>j$}\}.
\]
By the Leibniz rule, $A_D = D[x_1;\delta_1]\cdots[x_n;\delta_n]$, whence locally finite and the filtration defined above also defines a filtration on $A_D$.

(3) We have $\gr_{\cF}(A) = \kk[x_1,\hdots,x_n]$ and $\gr_{\cF}(A_D) = D[x_1,\hdots,x_n]$,
so clearly $\gr_{\cF}(A_D)$ is an order of $\gr_{\cF}(A)$.

(4) The algebra $\gr_{\cF}(A_D)$ is locally finite over $D$ and by \cite[Proposition 4.1]{ASZ} it is strongly noetherian.

(5) Let $F$ be a factor ring of $D$.
Then $A_D \tensor_D F = (D \tensor_D F)[x_1;\delta_1]\cdots[x_n;\delta_n]$ is noetherian.
Moreover, it is PI by \cite[Proposition 9]{LP}.
\end{proof}

The filtration that we put on $A$ in the previous lemma may not appear to be the most natural one. 
One must be careful because, to apply the lemma to pertinency, the automorphisms must respect the filtration. 
Thankfully this is the case for the algebras we consider.

Let $A=\kk[x][y;\delta]$ with $\delta(x)=p(x) \in \kk[x]$ and $\phi \in \Aut(A)$.
By \cite[Proposition 3.6]{AD2}, $\phi$ is triangular of the form
\[ \phi(x) = \lambda x + \kappa, \quad \phi(y) = \mu y + h(x),\]
$\lambda,\mu \in \kk^\times$, $\kappa \in \kk$, and $h(x) \in \kk[x]$ 
such that $p(\lambda x + \kappa)=\lambda\mu p(x)$.
In fact, adjusting the proof slightly above by setting $\deg_{\cF}(y) = \max\{p(x),h(x)\}$,
the argument goes through and we have that $\phi$ is a filtered automorphism.
Furthermore, $\phi$ restricts to a diagonal automorphism on $\gr_{\cF} A = \kk[x,y]$.

\begin{corollary}
\label{cor.diff}
Let $A=\kk[x][y;\delta]$ and let $G \subset \Aut(A)$ be a finite group that is small
when restricted to automorphisms of $\gr_{\cF} A = \kk[x,y]$.
Then $A\#G \iso \End_{A^G} A$.
\end{corollary}

\begin{proof}
The algebra $A$ is congenial by Theorem \ref{thm.iterated}.
Of course, $\kk[x,y]$ is CM and so $A$ is CM by \cite[Lemma 4.4]{SZ}.
It suffices to show that $\p(A,G)\geq 2$ by \cite[Theorem 4.10]{BHZ2}.
Then by the classical Auslander Theorem \cite{A} and \cite[Proposition 3.6]{BHZ2} we have
$\p(A,G) \geq \p(\gr(A),G) = 2$.
\end{proof}

\subsection*{Quantized Weyl algebras}
Let $T$ be an integral domain, $\bq \in M_n(T^\times)$ satisfying
$q_{ii}=1$ and $q_{ij}q_{ji}=1$ for all $i\neq j$,
and $\gamma=(\gamma_i) \in (T^\times)^n$.
The {\sf quantized Weyl algebra} $\qwa(T)$ is the algebra with basis $\{x_1,y_1,\hdots,x_n,y_n\}$ subject to the relations
\begin{align*}
	y_iy_j &= q_{ij} y_jy_i & & (\text{all } i,j) & x_iy_j &= q_{ji} y_jx_i & & (i < j) \\
	x_ix_j &= \gamma_i q_{ij} x_jx_i & & (i < j) & x_iy_j &= \gamma_j q_{ji} y_jx_i & & (i > j) \\
	x_jy_j &= 1 + \gamma_j y_jx_j + \sum_{l<j}(\gamma_l-1)y_lx_l & & 
			(\text{all } j).
\end{align*}
There is a standard filtration $\cF$ on $\qwa(T)$ defined by setting $\deg(x_i)=\deg(y_i)=i$.
Another filtration, which we do not consider here, sets all generators in degree one.

\begin{theorem}
\label{thm.qweyl}
The quantized Weyl algebra $\qwa(\kk)$ is congenial.
\end{theorem}

\begin{proof}
(1) The algebra $\qwa(\kk)$ is locally finite filtered using the filtration $\cF$ defined above.

(2) Let $D$ be the (finite) extension of $\ZZ$ by the $q_{ij}$
and the $\gamma_i$.
It is clear that $D$ is an order of $\qwa(\kk)$, $\qwa(\kk)_D = \qwa(D)$, and $\cF$ induces a filtration on $\qwa(D)$.

(3),(4) The associated graded algebra $\gr_{\cF}(\qwa(D))$ is just a quantum affine space
over $D$ and this is an order for the corresponding quantum affine space over $\kk$.
Viewing $\gr_{\cF}(\qwa(D))$ as an Ore extension gives the strongly noetherian condition \cite[Proposition 4.1]{ASZ}.

(5) Let $F$ be a factor ring of $D$ that is a finite field,
then $\qwa(D) \tensor_D F = \qwa(D \tensor_D F)$.
One can see that $\qwa(D \tensor_D F)$ is still affine noetherian by using the filtration.
That $\qwa(D \tensor_D F)$ is PI follows from \cite[Theorem 3.11]{GZ}.
\end{proof}

\subsection*{Filtered AS regular algebras}
A connected $\NN$-graded algebra $A$ is said to be {\sf Artin-Schelter (AS) regular} if
it has finite GK dimension, finite global dimension $d$,
and satisfies $\Ext_A^i(\kk,A) \iso \delta_{id} \kk(\ell)$, where $\delta_{id}$ is the Kronecker-delta and $\kk(\ell)$ is the trivial module $A/A_{\geq 1}$ shifted by some integer $\ell$.
An algebra $A$ is {\sf filtered AS regular} if $\gr_{\cF}(A)$ is AS regular
where $\cF$ is the standard filtration defined by setting all generators in degree one.

By \cite[Theorem 4.1]{CKWZ1}, the Auslander map is an isomorphism for the pair $(A,H)$
where $H$ is a semisimple Hopf algebra and $A$ is a graded $H$-module algebra that is
also a global dimension 2 Artin-Schelter (AS) regular algebra.
Our next goal is to extend this theorem to all filtered AS regular algebras
of global dimension 2.

The AS regular algebras of global dimension 2 are the {\sf quantum planes}
\[ \qp = \kk\langle x,y : xy-qyx\rangle, \quad q \in \kk^\times,\]
and the {\sf Jordan plane},
\[ \jp = \kk\langle x,y : yx-xy + y^2\rangle.\]
The filtered AS regular algebras of global dimension 2 \cite{twogen} include $\qp$, $\jp$,
$U(\fg)$ where $\fg$ is the two-dimensional solvable Lie algebra,
the {\sf quantum Weyl algebras}
\[ \wa = \kk\langle x,y : xy-qyx-1\rangle, \quad q \in \kk^\times,\]
and the {\sf deformed Jordan plane}
\[ \jp_1 = \kk\langle x,y : yx-xy + y^2+1\rangle.\]
Note that $\wa$ is a special case of the quantized Weyl algebras defined above.

\begin{remark}
\label{rem.newfilt}
The algebras $\jp$, $U(\fg)$, and $\jp_1$ are all examples of iterated 
differential operator rings over $\kk[x]$.
It is not difficult to check that the proof of Theorem \ref{thm.iterated} 
goes through when one replaces the filtration in that proof with the standard filtration.
\end{remark}


The notion of homological determinant was developed by J{\o}rgensen and Zhang, and we refer to \cite{gourmet} for the definition.
We denote the subgroup of homological determinant 1 automorphisms in $\Aut(A)$ by $\SL(A)$.

\begin{theorem}
\label{thm.filtAS}
Let $A$ be a filtered AS regular algebra of global dimension 2 and $G \subset \SL(A)$ a finite group of filtered automorphisms. Then $A\#G \iso \End_{A^G} A$.
\end{theorem}

\begin{proof}
First we consider graded algebras and graded group actions.
If $A=\kk[x,y]$, then this is Auslander's Theorem \cite{A}.
Now if $A=\jp$ or $\qp$ with $q \neq \pm 1$, 
then the result is true by \cite[Theorem 4.5]{MU}.
Note that this also proves the theorem for certain differential operator rings considered above.
The last case of $\qp$ with $q=-1$ follows from \cite[Theorem 4.1]{CKWZ1}.

Now we consider filtered actions.
The result for $U(\fg)$ is \cite[Theorem 0.4]{BHZ2} while $\jp_1$ is 
essentially Corollary \ref{cor.diff} but with $\gr\jp_1 = \jp$
(see Remark \ref{rem.newfilt}).
Finally, let $A=\wa$.
If $q=1$, then $A$ is the first Weyl algebra and this result follows from \cite[Theorem 2.4]{Mo}.
Otherwise, let $\cF$ be the standard filtration on $A$ so that $\gr_\cF(A) = \qp$. 
Then $A$ is congenial by Theorem \ref{thm.qweyl}.

As $\qp$ is CM \cite[Lemma]{LS}, then so is $A$ by \cite[Lemma 4.4]{SZ}.
It suffices to show that $\p(A,G)\geq 2$ by \cite[Theorem 4.10]{BHZ2}.
By \cite[Proposition 1.5]{AD1}, every automorphism of $A$ acts either
diagonally or anti-diagonally on the generators.
Thus, a finite subgroup $G \subset \Aut(A)$ restricts to 
a graded action on $\gr(A)=\qp$ and so
by \cite[Proposition 3.6]{BHZ2} and results above,
$\p(A,G) \geq \p(\gr(A),G) = 2$.
\end{proof}

\begin{corollary}
\label{cor.filtAS}
Let $A$ be a filtered AS regular algebra of global dimension 2
and $H$ a finite-dimensional semisimple Hopf algebra acting on $A$ inner-faithfully
with trivial homological determinant.
If the $H$-action preserves the filtration of $A$, then $A\# H \iso \End_{A^H} A$.
\end{corollary}

\begin{proof}
If $A$ is an AS regular algebra of global dimension 2, this is just \cite[Theorem 0.3]{CKWZ1}.
If $A$ is a non-PI filtered AS regular algebra of global dimension 2,
then by \cite[Theorem 0.1]{CWWZ} the action of $H$ factors through a group algebra and we may apply Theorem \ref{thm.filtAS}.
We are left to consider only $\wa$ with $q$ a root of unity.
Again by \cite[Proposition 3.6]{BHZ2}, $\p(A,H) \geq \p(\qp,H)$
and $\p(\qp,H) \geq 2$ by the above.
\end{proof}

\subsection*{Down-up algebras}
Let $\alpha,\beta,\gamma \in \kk$. The {\sf down-up algebra} $A(\alpha,\beta,\gamma)$
is the $\kk$-algebra generated by $d,u$ subject to the relations
\begin{align*}
d^2u &= \alpha dud + \beta ud^2 + \gamma d, \\
du^2 &= \alpha udu + \beta u^2d + \gamma u.
\end{align*}
We consider here only the case that $\beta\neq 0$,
which is equivalent to assuming that the algebra is noetherian \cite[Theorem 1]{KMP}.
We set $r,s$ to be the two roots of $x^2-\alpha x-\beta=0$.

\begin{theorem}
\label{thm.du}
Suppose $A(\alpha,\beta,\gamma)$ is a noetherian down-up algebra 
such that $r$ and $s$ are roots of unity of order at least 2.
Then $A(\alpha,\beta,\gamma)$ is congenial.
\end{theorem}

\begin{proof}
(1) The down-up algebra $A=A(\alpha,\beta,\gamma)$
is locally finite filtered with the standard
filtration $\cF$ defined by setting $\deg(u,v)=1$.

(2) Set $D=\ZZ[\alpha,\beta,\gamma,r,s]$.
Then $A_D$ is again a down-up algebra and the filtration $\cF$ induces a filtration on $A_D$.

(3) It is clear that $\gr_{\cF} A \iso A(\alpha,\beta,0)$,
whence $\gr_{\cF} A$ is again a down-up algebra.
Thus, $\gr_{\cF} A_D$ is an order of $\gr_{\cF} A$.

(4) Let $C$ be a commutative noetherian $D$-algebra.
Then $\gr_{\cF} A_D \tensor_D C$ is a (graded) down-up algebra over $C$, whence noetherian by \cite[Proposition 7]{Bav} and \cite[Corollary 2.2]{KMP}.
It is worth noting that \cite{KMP} only states the result for base ring a field,
but by appealing to the construction of a down-up algebras as a 
generalized Weyl algebra, it is sufficient by \cite[Proposition 7]{Bav} that the base ring $D$ is noetherian.
Hence, $\gr_{\cF} A_D$ is strongly noetherian.

(5) Let $F$ be a factor ring of $D$ that is a finite field.
Then $A_D \tensor_D F$ is still a down-up algebra (over F) and hence it is affine noetherian. 
By \cite[Theorems 4.2 and 4.4]{Hild}, some power of $d$ and $u$ are central and so $A_D \tensor_D F$
is module-finite over its center, whence PI.
\end{proof}

There are several other cases of down-up algebras, depending on the roots $r$ and $s$, where Theorem \ref{thm.du} applies.
This hypothesis was only used in (5) and so one needs only verify
that $A_D \tensor_D F$ is PI.
For instance, this still holds when $r\neq 1$ and $s=1$.
The interested reader is referred to \cite{Hild}.

\begin{corollary}
\label{cor.du}
Suppose $A=A(\alpha,\beta,\gamma)$ is a down-up algebra with 
$\beta \neq 0,-1$ or $(\alpha,\beta)=(\pm 2,-1)$, and $r,s$ roots of unity of order at least 2.
If $G$ is a finite group of filtered automorphisms of $A$ such that 
$\p(\gr_{\cF}A,G)\geq 2$, then $A\#G \iso \End_{A^G} A$.
\end{corollary}

\begin{proof}
By \cite[Lemma 4.2]{KMP}, $A$ and $\gr_{\cF}A$ are CM.
We may now apply Theorem \ref{thm.du}
\cite[Theorems 0.6, 4.10]{BHZ2} and \cite[Theorem 4.3]{GKMW}.
\end{proof}

\subsection*{Symplectic reflection algebras}
Let $V$ be a finite-dimensional vector space, 
$\omega$ a non-degenerate, skew symmetric bilinear form on $V$, 
$\Sp(V)$ the symplectic linear group on $V$,
and $G \subset \Sp(V)$ a finite group that is generated by $\cS$, the set of all symplectic reflections in G.
Given this data, along with $t \in \kk$ and a conjugate invariant function 
$\bc :\cS \rightarrow \kk$, i.e. $\bc(gsg\inv)=\bc(s)$ for all $s \in \cS$ and $g \in G$,
the {\sf symplectic reflection algebra} $\bH_{t,\bc}$ is defined as
\[
\bH_{t,\bc} = T(V^*) \# G / \left( u \tensor v - v \tensor u = t\omega(u,v) - \sum_{s \in \cS} \bc(s) \omega_s(u,v) : u,v \in V^* \right) 
\]
where $\omega_s$ is the 2-form on $V$ whose restriction to $\im(1-s)$ is $\omega$
and whose restriction to $\Ker(1-s)$ is zero.


\begin{theorem}
\label{thm.symplectic}
The symplectic reflection algebra $\bH_{t,\bc}$ is congenial.
\end{theorem}

\begin{proof}
Set $H=\bH_{t,\bc}$.

(1) There is a natural filtration $\cF$ on $H$ defined by setting $V^*$ in degree one
and $G$ in degree zero. The PBW theorem for symplectic reflection algebras 
\cite[Theorem 1.3]{EG} gives $\gr_{\cF} H \iso \kk[V]\# G$.
It follows from this that $H$ is a locally finite filtered algebra.

(2) Set $D$ to be the polynomial ring over $\ZZ$ generated by $t$ and $\bc(s)$ for all $s \in S$.
Then $D$ is an affine subalgebra of $\kk$ and it is clear that $H_D$ is an order of $H$.
The filtration defined above descends and so $H_D$ is locally finite filtered.

(3),(4) We have that $\gr_{\cF} H_D = (\gr_{\cF} H)_D = (\kk[V]\# G)_D$.
But $\kk[V]$ is congenial with $D$ and so $\kk[V]\# G$ is congenial by \cite[Proposition 4.8]{BHZ2}.
The filtration above is just the grading on $\kk[V]\# G$ with $V$ in degree 1 and $G$ in degree zero.

(5) Let $F$ be a finite field factor ring of $D$ and set $\ch F = p$.
Then $H_D \tensor F$ remains affine noetherian essentially by the PBW theorem.
By \cite[Appendix]{BFGE}, $H_D \tensor \kk_p$ is module-finite over its center whence PI, where $\kk_p$ is the algebraically closed field containing $F$.
It follows that the subring $H_D \tensor F$ is also PI.
\end{proof}


\begin{corollary}
Let $H=\bH_{t,\bc}$ be a symplectic reflection algebra and $K$ a semisimple Hopf 
algebra acting on $H$ preserving the filtration given above.
Then $H\#K \iso \End_{H^K} H$ if and only if $\p(H,K) \geq 2$.
\end{corollary}

\begin{proof}
We can directly apply \cite[Theorem 4.10]{BHZ2} so long as $H$ is CM,
but this follows from the filtration on $H$.
\end{proof}

\subsection*{Acknowledgments}
The authors thank James Zhang for several helpful conversations.
They also thank the anonymous referee for a careful reading and for pointing out a simplification of Definition \ref{defn.congenial}.

The first author was partially supported by a grant from the 
Miami University Senate Committee on Faculty Research.

\bibliographystyle{amsplain}

\providecommand{\bysame}{\leavevmode\hbox to3em{\hrulefill}\thinspace}
\providecommand{\MR}{\relax\ifhmode\unskip\space\fi MR }
\providecommand{\MRhref}[2]{%
  \href{http://www.ams.org/mathscinet-getitem?mr=#1}{#2}
}
\providecommand{\href}[2]{#2}

\end{document}